\documentclass[12pt]{amsart}

 \newtheorem{thm}{Theorem}[section]
 \newtheorem{cor}[thm]{Corollary}
 \newtheorem{lem}[thm]{Lemma}
 \newtheorem{prop}[thm]{Proposition}
 \theoremstyle{definition}
 
 \theoremstyle{remark}
 \newtheorem{rem}[thm]{Remark}
 \theoremstyle{definition}

 \newcommand{\PP}{\mathbb{P}}

\usepackage{xypic} \usepackage{amsthm}

\begin{document}

\title[Quadrics and linear spaces]{Subspace arrangements, configurations of linear spaces and the quadrics containing them}


\author[E. Carlini]{Enrico Carlini}
\address[E. Carlini]{Dipartimento di Matematica, Politecnico di Torino, Torino, Italia}
\email{enrico.carlini@polito.it}

\author[M.V.Catalisano]{Maria Virginia Catalisano}
\address[M.V.Catalisano]{DIPTEM - Dipartimento di Ingegneria della Produzione, Termoenergetica e Modelli
Matematici, Universit\`{a} di Genova, Piazzale Kennedy, pad. D 16129 Genoa, Italy.}
\email{catalisano@diptem.unige.it}

\author[A.V. Geramita]{Anthony V. Geramita}
\address[A.V. Geramita]{Department of Mathematics and Statistics, Queen's University, Kingston, Ontario, Canada, K7L 3N6 and Dipartimento di Matematica, Universit\`{a} di Genova, Genova, Italia}
\email{Anthony.Geramita@gmail.com \\ geramita@dima.unige.it  }


\date{}


\begin{abstract}
A subspace arrangement in a vector space is a finite collection of
vector subspaces. Similarly, a configuration of linear spaces in a
projective space is a finite collection of linear subspaces. In
this paper we study the degree 2 part of the ideal of such
objects. More precisely, for a generic configuration of linear
spaces $\Lambda$ we determine $HF(\Lambda,2)$, i.e. the Hilbert
function of $\Lambda$ in degree 2.
\end{abstract}

\maketitle

\section{Introduction}

If $V$ is an $n+1$ dimensional vector space then a {\it subspace arrangement}
is a finite collection of vector subspaces of $V$. This algebraic notion, and many properties of these objects, have been investigated from an algebraic point of view, see
\cite{Sidman04,Sidman07,BjornerPeevaSidman}. Their geometric counterparts are also of interest and are obtained by projectivizing all the vector spaces involved. Doing this one obtains, in $\PP^n=\PP(V)$, a finite collection of linear subspaces.  Such a collection is referred to as a {\em configuration of linear spaces}.

Subspace arrangements arise in many contexts and in many applications and
hence the same holds true for their alter ego, i.e. configurations of
linear spaces. Derksen \cite{Derksen} showed applications to
Statistics via Generalized Components Analysis. Also, Ma et al. gave applications to data modeling and segmentation in \cite{DerksenApplication}. Moreover, in \cite{MR2202248} configurations of linear spaces and their Hilbert
functions were shown to be related to the study of Segre-Veronese varieties and their higher
secant varieties. Another application, this time to incidence properties of
rational normal curves and linear spaces, is studied in \cite{CaCat07,CaCat09}. In this paper we exhibit yet another application, relating configurations of linear spaces to the study of a special kind of polynomial decomposition (for more on polynomial decompositions see \cite{Ca04JA} and \cite{Ca05Siena}).

As the Hilbert function of a configuration of linear spaces is of interest, we begin by recalling what is known about it. Derksen and Sidman \cite{DerksenSidman} have discovered many interesting results about the Castelnuovo-Mumford regularity (CM-regularity) of the intersection of ideals generated by generic linear forms. In our context these results give bounds on the CM-regularity for the ideal of a generic configuration of linear spaces $\Lambda$. Hence, one knows an integer $d_0$ such that for $d\geq d_0$ the Hilbert function and the Hilbert polynomial for the ideal of $\Lambda$ agree, i.e. $HF(\Lambda, d)=hp(\Lambda,d)$.
Moreover Derksen, in \cite{Derksen}, gives an explicit formula for $hp(\Lambda,d)$ when the configuration of linear spaces is generic.

Thus, for a generic configuration of linear spaces $\Lambda$, we need
only determine a finite number of values of
$HF(\Lambda,d)$ in order to have complete knowledge of the
Hilbert function of $\Lambda$. But, in general, this knowledge is available in only a few basic
situations, i.e. when $\dim\Lambda=0,1$. When $\dim\Lambda=0$ we are dealing with a generic set of points and for these the Hilbert function is known to be $HF(\Lambda,d)=\min\{{n+d \choose d},hp(\Lambda,d)\}$ for all $d$. When
$\dim\Lambda=1$ (where it is enough to consider a generic union only of lines i.e. we need not consider lines and points) the problem is considerably harder. The first
complete answer was given by Hartshorne and Hirschowitz in   \cite{HartshorneHirschowitz}.  There it is shown that if $\Lambda$ is a generic collection of lines in $\PP^n \ (n > 2)$ then
$$
HF(\Lambda, d) = \min \left \{{n+d\choose d}, hp(\Lambda,d) \ \right \} .
$$   The proof they present is long
and non-trivial and makes use of Castelnuovo's sequence (la methode d'Horace) and degeneration techniques. When $\dim\Lambda>1$ we are not aware of any general results in the
literature.

Inasmuch as the general problem of describing the Hilbert function of any generic configuration of linear spaces seems extremely difficult, one may consider some distinct subproblems.  One is to   consider a generic configuration of linear spaces (with no restriction on the dimension or on the number of components) and determine its Hilbert function in the first unknown degree, i.e. in degree 2. Another is to consider families of generic configurations of a special kind, e.g. impose bounds on the
dimension or on the number of components. The first subproblem  is
the subject of this paper.  We postpone a discussion of the second subproblem to another paper (see \cite{CarCatGer}) as it uses completely different techniques which are considerably inspired by
\cite{HartshorneHirschowitz} and involve multiple inductions
coupled with Castelnuovo's sequence.

In this paper we follow an approach which uses a fiber argument that reduces the
problem to a chain of numerical inequalities.  Using these ideas we
are able to completely determine the Hilbert function, in degree 2,
for a generic configuration of linear spaces (see Theorem
\ref{summarythm}).

The paper is organized in the following way.  Section 2 lays out some easy observations about certain Fano varieties.  Sections 3 and 4 contain the technical heart of the paper.  Section 3 deals with ``small" generic linear configurations, i.e. generic configurations for which the dimension of the components are such that no two components intersect.  A case by case argument shows that the number of quadrics containing such a generic configuration is exactly what one would expect.

Section 4 deals with ``large" generic configurations, i.e. generic configurations for which the dimensions of the components force (some) of the components to intersect.  There is again an ``expected" behavior, but it is more complicated to express.  We were able to calculate the Hilbert function in this case as well and thus show that the expression we had for the expected behavior was correct.

Finally, in Section 5, we give an application of our results to the problem of writing homogeneous polynomials in $\mathbb{C}[x_0, \ldots , x_n]$ as a sum of polynomials in fewer variables.

The first and second author wish to thank Queen's University, in
the person of the third author, for their kind hospitality during
the preparation of this work. All the authors enjoyed support
from NSERC (Canada) and GNSAGA of INDAM (Italy).  The first author
was also partially supported by a ``Giovani ricercatori, bando 2008"
grant of the Politecnico di Torino.

\section{Notation and preliminary results}

We say that  $\Lambda\subset\PP^n$ is a   {\it configuration of linear spaces} of $\Bbb P^n$ if  $\Lambda$ is a finite union of linear spaces.  We write
$$
\Lambda = \Lambda_1 + \dots + \Lambda_s \subset \Bbb P^n,
$$
where $\Lambda_i  \simeq \Bbb P^{m_i}$ is a  linear space of dimension $m_i$ ($0 \leq m_i < n$) and $\Lambda_i \neq \Lambda_j$ for $i\neq j$.  We set   $L(\Lambda) = (m_1, \dots, m_s)$ and call this the {\it weight vector} of $ \Lambda$.

Given a weight vector $L$, the configurations of linear spaces having weight vector $L$ are parameterized by
\[\mathcal{D}_L=\mathbb{G}(m_1,\PP^n) \times \ldots \times \mathbb{G}(m_i,\PP^n) \times \ldots \times \mathbb{G}(m_s,\PP^n).\]
We notice that
\[
\dim \mathcal{D}_L = \sum _{i=1} ^s (m_i+1)(n-m_i).\]

A configuration of linear spaces of $\Bbb P^n$ is {\it generic} if
its components $\Lambda_i$ are generic linear spaces of $\Bbb
P^n$, i.e. if $\Lambda = \Lambda_1 + \dots + \Lambda_s$ belongs to
a specified open not empty subset of $\mathcal{D}_L$.

Let $S=\mathbb{C}[x_0,\ldots,x_n] = \oplus_{i\geq 0}S_i$ denote the coordinate ring of
$\PP^n$. For any weight vector $L$ we consider the
following incidence correspondence
\[\Sigma_{L}=\left\lbrace(Q,\Lambda): Q\supset \Lambda\right\rbrace\subset  \PP(S_2)\times\mathcal{D}_L\]
and the projection maps
\[\phi_{L}:\Sigma_{L}\longrightarrow\PP(S_2),\psi_{L}:\Sigma_{L}\longrightarrow\mathcal{D}_L.\]
(Whenever $L$ is clear from the context we will drop the subscript $L$.)

To study the Hilbert function, in degree 2, of a generic configuration of linear spaces having weight vector $L$, we first notice that
$$
\dim (I_\Lambda)_2   =  \dim {\psi_{L}}^{-1}(\Lambda)+1
$$
where  $\Lambda\in\mathcal{D}_L$ is a generic point and
$(I_\Lambda)_2$ is the ideal of $\Lambda$ in degree 2.  So, if we know $\dim (I_\Lambda)_2$ we
can easily determine the Hilbert function of $\Lambda$ in degree 2 since $HF(\Lambda,2)={n+2\choose 2}-\dim (I_\Lambda)_2$.

\noindent{\bf Note:}\ We will often
prefer the ideal notation to the Hilbert function notation, as the
first will be easier to use in our situation.

Recall the following theorem (see  \cite[Theorem
22.13]{Harris}), that gives the dimension of the Fano variety of
$m$-planes on a smooth quadric hypersurface.

\medskip

\begin{thm}\label{Harris} The variety $F_{m,n-1}\subset \mathbb{G}(m, \PP^n)$ of $m$-planes on a smooth $(n-1)$-dimensional quadric hypersurface is smooth and if $m < \frac{n-1}{2}$ it is irreducible.

When $m \leq \frac{n-1}{2}$ it has dimension
$$
\dim F_{m,n-1} = (m+1) \left (n-1-{3m \over 2} \right ),$$
 and it is empty otherwise.
\end{thm}
\qed

\begin{cor}\label{corHarris}  The variety of  $m$-planes on a quadric $Q \subset \Bbb P^n $ of rank $r > 2n-2m$ is empty.
\end{cor}

\begin{proof}
For $r=n+1$ the quadric $Q$ is smooth and the conclusion follows from  Theorem \ref{Harris}.

Let $r<n+1$.  Recall  that $Q$ is the cone over a smooth quadric $\tilde
Q\subset \Bbb P^{n'} $, $n'= {r-1}$ having vertex  $V \simeq \PP^{n-r}$. Hence, each
$\Lambda \simeq \PP^m \subset Q$ is projected from the vertex $V$ into an $m'$-dimensional
linear subspace of $\tilde Q$, where
$$m' = m- \dim (\Lambda \cap V)-1 \geq  m - (n-r)-1.$$
Since
$$m' - {{n'-1} \over 2}  \geq  m - (n-r)-1 - {{(r-1)-1} \over 2} = {2m-2n+r \over 2} >0,
$$
by Theorem \ref{Harris},  there are no  $m'$-planes on $\tilde Q$, and the conclusion follows.

\end{proof}

\begin{lem}\label{lemma} Let $m_1 \geq \dots \geq m_s \geq 0$, and $m_1+m_2 <n$. Let
$$ \Lambda = \Lambda_1 + \dots + \Lambda_s \subset \Bbb P^n$$
be a generic configuration of linear spaces with weight vector
$L= (m_1, \dots, m_s)$ and let $r \leq 2n-2m_1$.

\medskip If $r$ is even (say $r = 2p$) and one has:

$${n+2\choose 2} - {n+2-r  \choose 2} \leq
\sum_{m_i\leq p-1} {m_i+2\choose 2}+ {p\over 2} \sum_{m_i\geq p} (2m_i+3-p)
, $$

or if $r$ is odd (say $r=2p+1$) and one has

$${n+2\choose 2} - {n+2 -r  \choose 2} \leq
$$
$$
\sum_{m_i\leq p-1} {m_i+2\choose 2} +{1 \over 2} \sum_{m_i\geq p}(2m_i+2+p(2m_i+1- p))
, $$

then $(I_{\Lambda})_2$ does not contain any rank $r$ quadric.
\end{lem}

\begin{proof}

We let $Q$ be a rank $r$ quadric in $I_\Lambda$.  We want to compute $\phi_L^{-1}(Q)$.  To do this we have to describe all the $\PP^{m_i}\subset Q$ for $i = 1, \ldots , s$. Since
 $Q$ is a cone over a smooth quadric $\tilde
Q\subset\PP^{r-1}$ having vertex a $\PP^{n-r}$, the generic
$\PP^{m_i} \subset Q$ is projected from the vertex onto an $m'_i$-dimensional
linear subspace of $\tilde Q$.

\medskip\medskip\noindent{\it Claim:}\ If $m_i \leq { {(r-1)-1}\over 2}$ then $m_i = m_i^\prime$.

\medskip\noindent{\it Proof of Claim:}\ Notice that $\tilde Q$ is smooth in a $\PP^{r-1}$.  By Theorem \ref{Harris} this implies that $m_i^\prime \leq  \frac{(r-1)-1}{2}$.  Hence, using the genericity of $\PP^{m_i}$ and the formula for the intersection of linear spaces of $\PP^n$ we get that $\PP^{m_i} \cap V$ has dimension
$$
\leq \max\left\{m_i +(n-r) - n ; 0\right\} = \max\{m_i-r; 0\}
$$
$$
\leq\max\{\frac{(r-1)-1}{2}-r; 0\} = 0
$$
Thus the projection of $\PP^{m_i}$ from $V$ maintains its dimension, i.e. $m_i^\prime = m_i$.

\medskip\medskip It follows from Theorem \ref{Harris} that the family of $m_i$-planes in $\tilde Q$ has dimension
\[
(m_i+1) \left (r-2-{3\over 2}m_i \right) .
\]

Moreover, the $\PP^{m_i} \subset Q$ projected onto the same $m_i$-plane form a
family of dimension
\[
(m_i+1)(n-r+1).
\]
(To see why this is so, note that these $m_i$-dimensional linear spaces all lie inside the linear space spanned by $V$ and $\PP^{m_i^\prime}$. That span has dimension $n-r + m_i^\prime + 1 = n-r+m_i+1$ and since $\dim\mathbb{G}(m_i, \PP^{n-r+m_i+1}) = (m_i+1)(n-r+1)$ the statement follows.)

\medskip For $m_i  > {r-2 \over 2} $,  by Theorem \ref{Harris},
$\tilde Q$ does not contain any $\PP^{m_i}$.
Hence $m'_i$ is  the
biggest natural number $m'$ such that there exists a $\PP^{m'}\subset\tilde Q$.
Clearly we have
$$
m' = \left \lfloor {r-2 \over 2} \right \rfloor .
$$
In this case the $\PP^{m'}\subset\tilde Q$ form a family of dimension
\[(m'+1) \left (r-2-{3\over 2}m'\right )\]
and for each of these $\PP^{m'}$, the family of $\PP^{m_i} \subset Q$
projecting onto it has dimension
\[(m_i+1)(n-r+m'-m_i+1).\]
Thus we get
$$
\dim \phi_L^{-1}(Q)=
\sum_{m_i \leq {r-2 \over 2} }
\left( (m_i+1) \left (r-2-{3\over 2}m_i \right)
+(m_i+1)(n-r+1) \right)
$$
$$
+
\sum_{m_i > {r-2 \over 2} }
((m'+1) \left(r-2-{3\over 2} m' \right)+(m_i+1)(n-r+m'-m_i+1))
$$
$$
=
\sum_{m_i \leq {r-2 \over 2} }
(m_i+1) \left (n-{3\over 2} m_i-1 \right )
$$
$$
+\sum_{m_i > \frac{r-2}{2}}
\left((m'+1) \left (r-1-{3\over 2} m' +m_i \right)+(m_i+1)(n-r-m_i) \right).
$$

Denoting by $\mathcal{Q}_r$ the family of rank $r$ quadrics we get
\[
\dim \phi_L^{-1}(\mathcal{Q}_r)={n+2\choose 2}-{n-r+2\choose 2}-1+\dim \phi_L^{-1}(Q) .
\]

It follows that
\[
\dim \phi_L^{-1}(\mathcal{Q}_r)-\dim\mathcal{D}_L
\]

\[
=\left[{n+2\choose 2}-{n-r+2\choose 2}-1+\dim \phi_L^{-1}(Q)\right]- \sum^s _{i=1} (m_i+1)(n-m_i).
\]
If $r$ is even (say $r=2p$ and hence $m' = p -1$) we obtain
$$
\dim \phi_L^{-1}(\mathcal{Q}_r)-\dim\mathcal{D}_L
$$

$$ =
{n+2\choose 2}-{n-2p+2\choose 2}-1+
\sum_{m_i \leq {p-1} }
(m_i+1) \left (n-{3\over 2} m_i-1 \right )
$$
$$ +
\sum_{m_i \geq {p} }
\left(p \left(2p+m_i- {3\over 2} p + {1\over 2}  \right) + (m_i +1 ) ( n-2p-m_i)
 \right)
 $$
 $$ - \sum^s _{i=1} (m_i+1)(n-m_i)
$$
$$ =
{n+2\choose 2}-{n-2p+2\choose 2}-1
-
\sum_{m_i\leq p-1} {m_i
+2\choose 2}- {p\over 2} \sum_{m_i\geq p} (2m_i+3-p).
$$

\medskip On the other hand, for $r$  odd (say $r = 2p + 1$ and so  $m' = p -1$) we obtain
$$
\dim \phi_L^{-1}(\mathcal{Q}_r)-\dim\mathcal{D}_L
$$

$$ =
{n+2\choose 2}-{n-(2p+1)+2\choose 2}-1+
\sum_{m_i \leq {p-1} }
(m_i+1) \left (n-{3\over 2} m_i-1 \right )
$$
$$ +
\sum_{m_i \geq {p} }
\left(p \left(2p+1+m_i- {3\over 2} p + {1\over 2}  \right) + (m_i +1 ) ( n-2p-1-m_i)
 \right)
 $$
 $$ - \sum^s _{i=1} (m_i+1)(n-m_i)
$$

$$ =
{n+2\choose 2}-{n-2p+1\choose 2}-1$$
$$
-\sum_{m_i\leq p-1} {m_i+2\choose 2} - {1 \over 2} \sum_{m_i\geq p}(2m_i+2+p(2m_i+1- p)).$$

Hence
\[\dim \phi_L^{-1}(\mathcal{Q}_r)-\dim\mathcal{D}_L<0,\]
and the conclusion follows.

\end{proof}

\section{Disjoint spaces}

We begin by considering configurations of linear spaces with non
intersecting components, i.e.
$$
\Lambda = \Lambda_1 + \cdots + \Lambda_s, \hbox{ and } \Lambda_i \cap \Lambda_j = \emptyset \hbox{ for } i \neq j.
$$
In this situation it is easy to guess the expected behavior: if too many components are
involved, then no quadric is expected to contain the configuration. We now give a precise statement and proof of this fact.

\begin{prop} \label{prop1}
Let $s \geq 2$, $m_1 \geq \dots \geq m_s \geq 0$, and $m_1+m_2 <n$ (hence $\Lambda_i \cap \Lambda_j = \emptyset$ for $i \neq j$).

Let
$$ \Lambda = \Lambda_1 + \dots + \Lambda_s \subset \PP^n$$
be a generic configuration of linear spaces having weight vector
$(m_1, \dots, m_s)$.

If
$$\sum _{i=1}^s {m_i+2 \choose 2}  \geq {n+2 \choose 2},
$$
then $\dim (I_{\Lambda})_2 =0$.

\end{prop}

\begin{proof}\ We will show that there are no quadrics of rank $r$ in $(I_{\Lambda})_2$ for $1 \leq r \leq n+1$ and that will certainly prove the proposition.

First observe that for $r>2n-2m_1$, the conclusion follows immediately from Corollary \ref{corHarris}.  So, we are reduced to considering the case where $r \leq 2n-2m_1$.

For $r$ in this range and $r$ even (say $r=2p$) set
\begin{equation} \label{definf(p)pari}
f(p) = \sum_{m_i\leq p-1} {m_i +2\choose 2}+ {p\over 2} \sum_{m_i\geq p} (2m_i+3-p).
\end{equation}

For $r$ in this range and $r$ odd (say $r=2p+1$) set
\begin{equation} \label{definf(p)disp}
f(p)= \sum_{m_i\leq p-1} {m_i+2\choose 2} + {1 \over 2} \sum_{m_i\geq p}(2m_i+2+p(2m_i+1- p)).
\end{equation}

Notice that these expressions for $f(p)$ are precisely the expressions that appeared in the inequalities of Lemma \ref{lemma}.  Exactly for that reason, if we can show that for $1 \leq r \leq \min \{ 2n-2m_1; n+1\}$ we have

\begin{equation} \label{f(p)diseg}
 {n+2\choose 2}-{n+2-r \choose 2} \leq f(p),
\end{equation}
the conclusion will follow immediately from Lemma \ref{lemma}.

First notice that an easy computation gives us that
\begin{equation} \label {induzionexppari}
f(p+1)-f(p) =  \sum_{m_i\geq p} (m_i-p+1)  \ \ \ \ \ \ {\rm when } \ r=2p ,
\end{equation}

and

\begin{equation} \label {induzionexpdispari}
f(p+1)-f(p) =  \sum_{m_i\geq p} (m_i-p)  \ \ \ \ \ \ {\rm when } \ r=2p+1.
\end{equation}

\noindent{\it Case 1}: $r=2p$ and  $n \leq 2m_1$.

\medskip

In this case $2n-2m_1 <n+1$ and hence $\min\{2n-2m_1; n+1\} = 2n-2m_1$. So, it will be enough to  prove (\ref{f(p)diseg}) for $1 \leq r \leq 2n-2m_1$, that is for $1 \leq p \leq n-m_1$.

Notice that we have
$m_1 -p \geq m_1 - n+m_1 \geq 0$, and hence $m_1 \geq p$.

\medskip

We proceed by induction on $n-m_1-p$.

If $n-m_1-p = 0$ then $p = n-m_1$.  Since $m_1 \geq p$ we have $m_2 \leq n-m_1-1 =p-1$.

Hence,
recalling  that $\sum _{i=1}^s {m_i+2 \choose 2}  \geq {n+2 \choose 2},
$ we have
$$
f(n-m_1) - {n+2\choose 2}+{n+2-2p \choose 2}
$$
$$=
 \sum_{i\geq 2} {m_i +2\choose 2}+  \frac{n-m_1}{2}  (2m_1+3-n+m_1)
$$
$$
  - {n+2\choose 2}+{2m_1-n+2 \choose 2}
$$
$$
\geq  -{m_1 +2\choose 2}+ {{n-m_1}\over 2}  (3m_1+3-n)
 +{2m_1-n+2 \choose 2} =0.
$$
This finishes the case $n-m_1-p = 0$.

\medskip

Now suppose that $n-m_1-p >0$, i.e.  $p < n-m_1$. By (\ref{induzionexppari}) we have
$$
f(p)-{n+2\choose 2} + {n+2-r\choose 2}  =
$$
$$
 f(p+1)-\sum_{m_i\geq p}(m_i - p + 1) - {n+2\choose 2} + {n+2-2p\choose 2}   \eqno{(*)}
$$

By induction
$$
f(p+1)\geq {n+2\choose 2} - {n+2-2(p+1) \choose 2}  .
$$

Hence, induction give us that
$$
(*) \geq  {n+2\choose 2}- {n-2p\choose 2} - \sum_{m_i\geq p} (m_i-p+1)  - {n+2\choose 2}+{n+2-2p \choose 2}
$$
$$
= (2n-4p+1)- \sum_{m_i\geq p} (m_i-p+1)  .
$$

We have thus proved
$$
f(p)-{n+2\choose 2} + {n+2-r\choose 2} \geq (2n-4p+1) - \sum_{m_i\geq p} (m_i-p+1) .
$$
Clearly if $\sum_{m_i\geq p} (m_i-p+1) \leq 2n-4p+1$, we get
$$
f(p)-{n+2\choose 2} + {n+2-r\choose 2}  \geq 0,
$$
and we are done.

It remains to see when this does not happen.  That is the content of the following Claim.

\medskip\noindent{\it Claim:}\ Continuing with the hypothesis of Case 1 as well as the induction hypothesis that $n-m_1-p > 0$, if
$$
(2n-4p+1) - \sum_{m_i\geq p} (m_i-p+1) <0
$$
then $m_4 \geq p$.

\noindent{\it Note}: Once this claim is proved we need only show that (\ref{f(p)diseg})
holds if $m_4 \geq p$.

\medskip\noindent{\it Proof of Claim:}\ Since $m_1 \geq p$ we first consider the possibility that $m_2 < p$.  In this case,
$$
(2n-4p+1) - \sum_{m_i\geq p} (m_i-p+1) = (2n-4p+1) - (m_1-p+1) = 2n-m_1-3p .
$$
Since $p \leq n-m_1-1$ we have
$$
2n-m_1-3p \geq 3-n+2m_1 > 0
$$
which is a contradiction.

Now suppose $m_1\geq m_2\geq p$ and $m_3<p$.  In like manner we get that
$$
(2n-4p+1) - \sum_{m_i\geq p} (m_i-p+1) = 2n-2p-m_1-m_2-1 .
$$
As above
$$
2n-2p-m_1-m_2-1 \geq m_1-m_2 + 1\geq 1 > 0
$$
again giving a contradiction.

Finally, suppose that $m_1\geq m_2\geq m_3\geq p$ and $m_4 < p$.  Then
$$
(2n-4p+1) - \sum_{m_i\geq p} (m_i-p+1) = 2n - p - 2 - m_1 - m_2 - m_3 .
$$
As before, with $p \leq n-m_1 -1$ and, in addition, using  $n \geq m_1 + m_2 + 1$, we get that
$$
(2n-4p+1) - \sum_{m_i\geq p} (m_i-p+1)\geq m_1-m_3 \geq 0
$$
which is again a contradiction.

Thus, $m_4 \geq p$ as we wanted to show.

\medskip\medskip  We  return to the induction proof.  However, now we can also assume that
\begin{equation} \label{diseg1}
  2n-4p+1 - \sum_{m_i\geq p} (m_i-p+1) < 0
\end{equation}
and hence that $m_4 \geq p$.

It follows from (\ref{definf(p)pari}) and  (\ref{diseg1}) that
$$f(p) - {n+2\choose 2}+{n+2-2p \choose 2}$$
$$=
 \sum_{m_i\leq p-1} {m_i +2\choose 2}+ {p\over 2} \sum_{m_i\geq p} (2m_i+3-p)
 +p(2p-2n-3)
$$
$$
\geq
{p\over 2} \sum_{m_i\geq p} (2m_i+3-p)+p(2p-2n-3)
$$
$$={p\over 2} \sum_{m_i\geq p} (2 (m_i-p+1)+(p+1))+p(2p-2n-3)
$$
$$\geq {p}(2n-4p+2)+{p\over 2}\cdot 4(p+1)+p(2p-2n-3) =p \geq 0.
$$

\medskip

\noindent{\it Case 2}: $r=2p$ and  $n \geq 2m_1+1$.  (This is very similar to Case 1 and so we will omit many details.)

In this case $2n-2m_1 \geq n+1$, so we  will prove (\ref{f(p)diseg}) for $1 \leq r \leq n+1$, that is
for $1 \leq p \leq \left \lfloor  {n+1 \over 2 } \right  \rfloor$.

 We proceed by induction on $\left \lfloor  {n+1 \over 2 } \right  \rfloor -p$.

 If $p = \left \lfloor  {n+1 \over 2 } \right  \rfloor$,  since $n \geq 2m_1+1$, we have $m_1 \leq p-1$, hence  $f \left( \left \lfloor  {n+1 \over 2 } \right  \rfloor \right) =
 \sum _{i=1}^s {m_i+2 \choose 2}
 $ and since
 $\sum _{i=1}^s {m_i+2 \choose 2}  \geq {n+2 \choose 2}
$, then  (\ref{f(p)diseg}) holds for $p= \left \lfloor  {n+1 \over 2 } \right  \rfloor $.

Now assume $p < \left \lfloor  {n+1 \over 2 } \right  \rfloor$.
As in the previous case,  if
$$
\sum_{m_i\geq p} (m_i-p+1) \leq 2n-4p+1,
$$
then (\ref{f(p)diseg}) holds,
hence let
\begin {equation} \label {diseg2}\sum_{m_i\geq p} (m_i-p+1) > 2n-4p+1,
\end{equation}
and note that, since $p < \left \lfloor  {n+1 \over 2 } \right  \rfloor$, then  $2n-4p+1 \geq 0$, so at least
 $m_1 \geq p$.
If $m_i \geq p$ we  have
\begin {equation} \label {diseg3}
 (2m_i+3-p) \geq
 (m_i-p+1) \cdot  \frac{2m_1+3-p}{m_1-p+1},
\end{equation}
in fact
$$ (2m_i+3-p)(m_1-p+1) - (m_i-p+1)(2m_1+3-p)
$$
$$= (m_1-m_i)(p+1) \geq 0.
$$
Thus by (\ref {diseg2}) and (\ref {diseg3})  we have
$$f(p) =
\sum_{m_i\leq p-1} {m_i +2\choose 2}+ {p\over 2} \sum_{m_i\geq p} (2m_i+3-p)
$$
$$
\geq  {p\over 2} \sum_{m_i\geq p} (m_i-p+1) \cdot  {2m_1+3-p \over m_1-p+1}
$$
$$
\geq  {p\over 2} (2n-4p+2) \cdot  {2m_1+3-p \over m_1-p+1}.
$$
Since   ${n+2\choose 2}-{n+2-2p \choose 2} = p(2n+3-2p) $
if we prove that
\begin {equation} \label {diseg4} {p} (n-2p+1) \cdot  {2m_1+3-p \over m_1-p+1} \geq p(2n+3-2p),
\end{equation}
then (\ref{f(p)diseg}) holds.  By easy computation we have that  (\ref {diseg4}) holds if and only if
$$p(n-2m_1-1)-p+n-m_1 \geq 0,
$$
so, since $p \leq n-m_1$ and $n \geq 2m_1+1$, the conclusion follows.

\medskip

{\it Case 3}: $r=2p+1$ and  $n \leq 2m_1+1$.

In this case $2n-2m_1 \leq n+1$, thus we will prove (\ref{f(p)diseg}) for $1 \leq r \leq 2n-2m_1$, that is
for $0 \leq p \leq n-m_1-1$.

We have
$m_1 -p \geq m_1 - n+m_1+1 \geq 0,
$
hence $m_1 \geq p$.

 We will work by induction on $n-m_1-1-p$.

 Let  $p = n-m_1-1$. Since $n\geq m_1+m_2+1$, we have $m_2 \leq p$. Let $\alpha$ denote the number of $m_i$ equal to $ p$.

We have  (recalling the definition  (\ref{definf(p)disp}) and that  $\sum _{i=1}^s {m_i+2 \choose 2}  \geq {n+2 \choose 2}$)
$$ f(p) - {n+2\choose 2}+{n+1-2p \choose 2}$$
$$ \geq -  \sum_{m_i\geq p} {m_i+2\choose 2} + {1 \over 2} \sum_{m_i\geq p}(2m_i+2+p(2m_i+1- p))
  +{n+1-2p \choose 2}$$
  $$= -{m_1+2\choose 2} - \alpha {p+2\choose 2} + m_1 +1 +{p\over2}  (2m_1+1- p)+
  \alpha ( p+1+{p\over2} (2p+1-p))
  $$
  $$+{n+1-2p \choose 2}
   =  m_1 -p +1.
  $$
Hence, since $p =n-m_1-1$,  and $n \leq 2m_1+1$, we get  $m_1 -p +1 \geq 1$, and
(\ref{f(p)diseg}) holds for $p= n-m_1-1.$

Now let $p \leq n-m_1-2$.

By (\ref{induzionexpdispari}) and by the inductive hypothesis we have
$$f(p) - {n+2\choose 2}+{n+1-2p \choose 2}$$
$$=f(p+1) - \sum_{m_i\geq p} (m_i-p)  - {n+2\choose 2}+{n+1-2p \choose 2}
$$
$$\geq  {n+2\choose 2}- {n-2p-1\choose 2} - \sum_{m_i\geq p} (m_i-p)  - {n+2\choose 2}+{n+1-2p \choose 2}
$$
$$= 2n-4p-1- \sum_{m_i\geq p} (m_i-p)  .
$$
Hence for $\sum_{m_i\geq p} (m_i-p) \leq 2n-4p-1$, we obtain the conclusion.

Let
\begin{equation} \label{diseg1disp}
\sum_{m_i\geq p} (m_i-p) \geq 2n-4p,
\end{equation}
and let $\alpha$ denote  the number of $m_i$ equal or bigger than  $ p$.
Recalling that $m_1+m_2+1 \leq n \leq 2m_1+1$ and $p \leq n-m_1-2$,  it is easy to show that
$\alpha \geq 5.$
Hence
 by (\ref{definf(p)disp}),  (\ref{diseg1disp}) and by eliminating the $m_i <p$, we have
$$f(p) - {n+2\choose 2}+{n+1-2p \choose 2}$$
$$
\geq   {1 \over 2} \sum_{m_i\geq p}(2m_i+2+p(2m_i+1- p)) + 2p^2 -2np-p-n-1
$$
$$
=
 \sum_{m_i\geq p}((m_i-p)(p+1)+ {1 \over 2} (p^2+3p+2)) + 2p^2 -2np-p-n-1
$$
$$
\geq (2n-4p)(p+1)+ {\alpha   \over 2} (p^2+3p+2)+ 2p^2 -2np-p-n-1
$$
$$
\geq n-2p^2-5p-1+ {5  \over 2} (p^2+3p+2) >0.
$$
\medskip

{\it Case 4}: $r=2p+1$ and  $n \geq 2m_1+2$.

In this case $2n-2m_1 \geq n+1$, so we  will prove (\ref{f(p)diseg}) for $1 \leq r \leq n+1$, that is
for $0 \leq p \leq \left \lfloor  {n \over 2 } \right  \rfloor$.
As in the previous cases,   we will work by induction. For $p =  \left \lfloor  {n \over 2 } \right  \rfloor $,
  since $n \geq 2m_1+2$, we have $m_1 \leq p-1$, hence  $f \left( \left \lfloor  {n \over 2 } \right  \rfloor \right) =
 \sum _{i=1}^s {m_i+2 \choose 2}
 $ and since
 $\sum _{i=1}^s {m_i+2 \choose 2}  \geq {n+2 \choose 2}
$, then  (\ref{f(p)diseg}) holds for $p= \left \lfloor  {n \over 2 } \right  \rfloor $.

Let   $p < \left \lfloor  {n \over 2 } \right  \rfloor$.

As in Case 3,  if $\sum_{m_i\geq p} (m_i-p) \leq 2n-4p-1$,  then (\ref{f(p)diseg}) holds.
Let
\begin {equation} \label {diseg2disp}\sum_{m_i\geq p} (m_i-p) \geq 2n-4p,
\end{equation}
and note that, since $p \leq \left \lfloor  {n \over 2 }  \right  \rfloor -1$, then  $2n-4p > 0$, so at least
 $m_1 > p$.
For  $m_i \geq p$ we  have
\begin {equation} \label {diseg3disp}
 2m_i+2+p(2m_i+1- p) \geq
 (m_i-p) \cdot  {2m_1+2+p(2m_1+1-p) \over m_1-p
 },
\end{equation}
in fact
$$
(2m_i+2+p(2m_i+1- p))( m_1-p) -  (m_i-p)  ( 2m_1+2+p(2m_1+1-p) )$$
$$= (m_1-m_i)(p^2+3p+2) \geq 0.
$$
Omitting the $m_i <p$, by (\ref {diseg2disp}) and (\ref {diseg3disp})  we obtain
$$f(p) \geq
 {1 \over 2} \sum_{m_i\geq p}(2m_i+2+p(2m_i+1- p))
$$
$$ \geq  {1 \over 2}  \sum_{m_i\geq p} (m_i-p) \cdot  {2m_1+2+p(2m_1+1-p) \over m_1-p },
$$
$$ \geq  {1 \over 2} (2n-4p) {2m_1+2+p(2m_1+1-p) \over m_1-p }.
$$
Since   ${n+2\choose 2}-{n+1-2p \choose 2} = 2np-2p^2+n+p+1$
if we prove that
$$
 (n-2p) {2m_1+2+p(2m_1+1-p) \over m_1-p } \geq 2np-2p^2+n+p+1,
$$
we obtain the conclusion. Now, since $n \geq 2m_1+2$ and $m_1 \geq p$, we easily get
$$
 (n-2p) (2m_1+2+p(2m_1+1-p))- ( m_1-p )( 2np-2p^2+n+p+1)
 $$
 $$=
n(m_1+p^2+2p+2)-5m_1p-3p-2m_1p^2-p^2-m_1
$$
$$\geq
2(m_1+1)(m_1+2)+(p+1)(p-m_1)
$$
$$\geq
2(p+1)(m_1+2)+(p+1)(p-m_1) =(p+1)(m_1+p+4) > 0,
$$
and the conclusion follows.

\end{proof}

The previous proposition describes the behavior of configurations
of linear spaces with non intersecting components. More precisely,
we know that if no quadric is expected then no quadric containing
the configuration exists. Using this knowledge we now show that
generic configurations having disjoint components always have the
expected behavior.

\begin{thm} \label{teorema1}
Let $s \geq 2$, $m_1 \geq \dots \geq m_s \geq 0$, and $m_1+m_2 <n$ and so $\Lambda_i \cap \Lambda_j = \emptyset$ for $i \neq j$.

Let
$$ \Lambda = \Lambda_1 + \dots + \Lambda_s \subset \PP^n$$
be a generic configuration of linear spaces having weight vector
$(m_1, \dots, m_s)$.
Then
$$\dim (I_{\Lambda})_2 = \max \left \{  {n+2\choose 2} - \sum _{i=1}^s {m_i+2 \choose 2} ; 0
\right \} $$

\end{thm}

\begin{proof}
For $\sum _{i=1}^s {m_i+2 \choose 2}  \geq {n+2 \choose 2}$ the conclusion follows from Proposition
\ref{prop1}.

If  $\sum _{i=1}^s {m_i+2 \choose 2}  < {n+2 \choose 2}$,
let $ \tilde \Lambda$ be the configuration  of linear spaces obtained by adding $x=  {n+2 \choose 2}-\sum _{i=1}^s {m_i+2 \choose 2} $ generic points to $\Lambda$, that is,
$$ \tilde \Lambda = \Lambda_1 + \dots + \Lambda_s+\Lambda_{s+1}+  \dots + \Lambda_{s+x},$$
where
 $\Lambda_{s+1}, \dots, \Lambda_{s+x}$
are $x$ generic points.
By applying Proposition \ref{prop1} to the configuration
$ \tilde \Lambda ,$
of weight vector $(m_1, \dots, m_s,0, \dots,0)$,
we have $\dim (I_{\tilde \Lambda})_2 =0$, hence, since the $x$ points are generic, we obtain
 $\dim (I_{\Lambda})_2 =x$.

\end{proof}
From  Theorem \ref{teorema1} we obtain
the following easy remark, which we will use in Section 4.:

\begin{rem} \label{4spaziuguali}
If $\Lambda \subset \Bbb P^n$ is  a generic configuration of four
linear spaces with  the same dimension $m$, and  $n=2m+1$, then
$\dim (I_{\Lambda})_2 =0$. In fact
$$
 {n+2 \choose 2} - \sum _{i=1}^4 {m+2 \choose 2} =-m-1<0.
$$
Moreover, if $\Lambda \subset \tilde \Lambda \subset \Bbb P^n$, then obviously $\dim (I_{\tilde \Lambda})_2 =0$.

\end{rem}

\section{Intersecting spaces}

In this section we deal with the case of generic configurations of
linear spaces for which some of the components intersect. The
configurations are still generic but now intersections arise
because of dimension reasons. In this situation it is harder to
express the notion of ``expected behavior" in simple terms.

Before stating the Main Theorem of this section, we want to make an easy, but useful, observation about projections of families of quadrics which contain a common linear space in their vertex.

\begin{rem} \label{rem} Let
$$ \Lambda = \Lambda_1 + \dots + \Lambda_s \subset \Bbb P^n$$
be a generic configuration of linear spaces having weight vector
$$ L = (m_1, \dots, m_s).$$ Assume that the  forms in $(I_{\Lambda})_2$ define cones,
each of which has vertex  containing a fixed linear space, $V$, of dimension $d$.  Consider the projection from $V$ onto a generic linear space $H \subset\Bbb P^n$ of complementary dimension $n^\prime = n-d-1$.  Then  each $\Lambda _i$ is projected from $V$ onto a linear space $\Lambda' _i \subset H$ where
$$
\dim \Lambda'_i := m_i^\prime  = {m_i - \dim (V \cap \Lambda_i) -1}
$$
(if  $V \cap \Lambda_i = \emptyset $,  we use the convention that $\dim (V \cap \Lambda_i) = -1$ ).

We have
$$ \dim  (I_{\Lambda})_2 = \dim (I_{\Lambda '})_2$$
where $$\Lambda '  =\Lambda'_1 + \dots + \Lambda'_s \subset  H \simeq \PP^{n^\prime} = \Bbb P^{n-d-1}$$
 is a generic configuration of linear spaces and $ I_{\Lambda '}$ is the ideal of $\Lambda^\prime$ in the homogeneous coordinate ring of $\PP^{n\prime}$.

\end{rem}

We now state and prove the Main Theorem of this section.

\begin{thm}\label{teorema2}

Let $m_1 \geq \dots \geq m_s \geq 0$, and $m_1+m_2 \geq n$. Let
$$ \Lambda = \Lambda_1 + \dots + \Lambda_s \subset \Bbb P^n$$
be a generic configuration of linear spaces having weight vector
$(m_1, \dots, m_s)$, and let
$$\tau = \max \{ i \in \Bbb N | m_1+m_i \geq n \} ,
$$
and
$$
v  = \sum _{i=2}^{\tau} (m_1+m_i -n+1).
$$
Then the following statements hold.
\begin{enumerate}
\item[(i)] If $v \geq m_1+1$, then $\dim (I_{\Lambda})_2 =0$.

\item[(ii)] If $v$ is such that  $2m_1-n+2 \leq v \leq m_1$ and
\begin{enumerate}
\item if  $\tau \geq 4$, then $\dim (I_{\Lambda})_2 =0$;

\item if  $\tau =3$, $s \geq 4$, $2n \leq  \sum _{i=1}^{4} m_i +2$, then $\dim (I_{\Lambda})_2 =0$;

\item  if  $\tau =3$,  $s \geq 4$ and $2n \geq  \sum _{i=1}^{4} m_i +2$,
or $\tau =s=3$,  then
$$\dim (I_{\Lambda})_2 $$
$$=\max  \left \{  {n+2\choose 2} - \sum _{i=1}^s {m_i+2 \choose 2}
+\sum _ {i,j=1,2,3} {m_i+m_j-n+2 \choose 2}  ; 0
\right \} .$$

\end{enumerate}

\item[(iii)] If $v \leq 2m_1-n+1 $,   then
$$\dim (I_{\Lambda})_2 $$
$$=\max  \left \{  {n+2\choose 2} - \sum _{i=1}^s {m_i+2 \choose 2}
+\sum _{i=2}^{\tau}  {m_1+m_i-n+2 \choose 2}  ; 0
\right \} .$$
\end{enumerate}

\end{thm}

\begin{proof} \ We fix the following notation:
$$
\Lambda_{i,j} = \Lambda_i \cap \Lambda_j \hbox{ and } V = \langle \Lambda_{1,2},\Lambda_{1,3}, \dots, \Lambda_{1,\tau} \rangle \subset \Lambda_1 .
$$
Since $v= \sum _{i=2}^{\tau} (\dim \Lambda _{1,i}+1)$ and the $\Lambda_i$ are generic linear spaces, then the linear span $V$ has dimension
\begin {equation} \label{span}
\dim V = \min \{ v-1; m_1\},
\end {equation}
and the points of $V$ are singular points for the quadrics defined by the forms of $I_{\Lambda}$.  Hence the quadrics  through $\Lambda$ are cones whose vertex contains $V$.

\medskip

(i) \   Since $v \geq m_1 + 1$ we must have $\dim V = m_1$. Projecting from $V$, and using the notation of Remark \ref{rem}, we have $n' = n-m_1-1$ and  $$m'_2=m_2-\dim \Lambda_{1,2} -1= m_2 - (m_1+m_2-n)-1= n-m_1-1.$$  Hence $H = \Lambda_2^\prime$ and so $(I_{\Lambda^\prime})_2 = (0)$ and, by Remark \ref{rem}, so is $(I_\Lambda)_2$.

\medskip

(ii) \  First note that  if $\tau=2$ , since $v \geq 2m_1-n+2 $, we have a contradiction:
$$0 \leq v-2m_1+n-2= m_1 + m_2- n+ 1 -2m_1+n-2=-m_1+m_2-1 \leq -1.
$$
Hence $\tau \geq 3$.

(a)  By (\ref{span}) we have $\dim V = v-1$. For $2\leq i \leq \tau$, we have $V \cap \Lambda_i =   \Lambda_1 \cap \Lambda_i $, so  $\dim V \cap \Lambda_i =m_1+m_i-n$, and for $i > \tau$, $V \cap \Lambda_i = \emptyset$. By projecting from $V$ into $H $ we obtain
$$ \Lambda' = \Lambda'_1 +\dots + \Lambda' _s  \subset H \simeq  \Bbb P^{n'}; \ \ \ \
 \Lambda'_i \simeq  \Bbb  P^{m_i'}
$$
$$n'= n-v; \ \ \ m'_1=  m_1-v;
$$
$$ m'_i=  n-m_1-1,\ \ {\rm for} \ \ 2 \leq i \leq \tau; \ \ \ \ m'_i=  m_i,\ \ {\rm for} \ \  i > \tau.$$

For $i=2,3,4$ we have $ n'-m'_1-m'_i =n-v- (m_1-v)-(n-m_1-1)=1$, hence
$\Lambda'_1 \cap \Lambda'_i  = \emptyset .$

Since $v \geq 2m_1-n+2$,  we have
$$
\dim (\Lambda'_2 \cap \Lambda'_3) = m'_2+m'_3-n' =v-2m_1+n-2  \geq 0,
$$
 hence
$\Lambda'_2 \cap \Lambda'_3 \neq \emptyset$.

We again apply Remark \ref{rem}, but this time to $\Lambda^\prime$ i.e. by projecting $\Lambda'$ from $\Lambda'_2 \cap \Lambda'_3$ into a linear space $H' \simeq \Bbb P^{ n''}$, we get

$$ \Lambda'' = \Lambda''_1 +\dots + \Lambda'' _s  \subset H' \simeq  \Bbb P^{n''}; \ \ \ \  \
\Lambda''_i \simeq  \Bbb P^{m_i''};
$$
$$
n''= n'-\dim \Lambda'_2 \cap \Lambda'_3 -1=2m_1-2v+1; $$
$$m''_1=  m_1-v;
$$
$$ m''_2=m''_3=  n-m_1-1-(v-2m_1+n-2+1)=m_1-v ;
$$
$$m''_4=  n-m_1-1- \dim \Lambda'_4 \cap ( \Lambda'_2 \cap \Lambda'_3)-1 \geq m_1-v.
$$
Hence by Remark \ref {4spaziuguali}
 it follows that $\dim (I_{\Lambda})_2= \dim (I_{\Lambda''})_2=0.$
  \medskip

(b),(c) \ In these cases $v \geq 2m_1-n+2$ implies
 $$ 2m_1+m_2+m_3 -2n+2 -2m_1+n-2 \geq 0, $$
 that is $m_2+m_3 \geq n$. It follows that $\Lambda_2 \cap \Lambda_3 \neq \emptyset$.

 Let $W = <\Lambda_{1,2}, \Lambda_{1,3}, \Lambda_{2,3}>$.  By (\ref{span}),
 $$\dim <\Lambda_{1,2}, \Lambda_{1,3}> =  \dim \Lambda_{1,2} +\dim \Lambda_{1,3}+1
 = 2m_1+m_2+m_3-2n+1 .$$
 Since $v \leq m_1$, we have
 $ 2m_1+m_2+m_3 -2n+2 -m_1\leq 0$, and from this inequality it follows that
 $\Lambda_{2,3} \cap \Lambda_1 = \emptyset$. Thus
$$\dim W = \dim \Lambda_{1,2} +\dim \Lambda_{1,3}+ \dim  \Lambda_{2,3} +2
$$
$$
=  2m_1+2m_2+2m_3-3n+2.
$$

  By projecting $\Lambda$ from
$W$ into a linear space $H \simeq \Bbb P^{n'}$ (following Remark \ref{rem} using $W$ instead of $V$), we get

$$ \Lambda' = \Lambda'_1 +\dots + \Lambda'_s  \subset H \simeq  \Bbb P^{n'}; \ \ \ \  \
\Lambda'_i \simeq  \Bbb P^{m_i'};
$$
$$
n'= n-\dim W -1=4n-2m_1-2m_2-2m_3-3; $$
$$m'_1=  m'_2 = m'_3  = 2n-m_1-m_2-m_3-2
$$
and, for  $s \geq4$, we have also :
$$m'_4=  m_4- \dim \Lambda_4 \cap W -1.
$$
Note that for $s=3$ we may apply Theorem \ref{teorema1}.

 Let $s \geq4$.
Now, $\Lambda_4 \cap W \neq \emptyset$ if and only if
$$m_4 + 2m_1+2m_2+2m_3-3n+2 \geq n,
$$
hence for $m_4 \geq 4n -2 m_1-2m_2-2m_3-2$ we have
$$m'_4 = m_4 - (m_4 + 2m_1+2m_2+2m_3-3n+2 -n)-1= n' ,
$$  that is $\Lambda'_4 \simeq \Bbb P^{n'}$, and  $\dim (I_{\Lambda})_2= \dim (I_{\Lambda'})_2=0$ immediately follows.

For $m_4 < 4n -2 m_1-2m_2-2m_3-2$ we have
$m'_4=  m_4 $. Then in case (ii)(b), where  $2n \leq  \sum _{i=1}^{4} m_i +2$, we obtain
$$m'_4\geq m_1'$$
Since $n' = 2m_1'+1$, by Remark \ref {4spaziuguali} we have
$\dim (I_{\Lambda})_2= \dim (I_{\Lambda'})_2=0,$ and this completes the proof of case (ii)(b).

If we are in case (ii)(c) we have $2n >  \sum _{i=1}^{4} m_i +2$ hence $m_1+m_2+m_3 <2n-2-m_4$. It follows that  $m_4 -4n +2 m_1+2m_2+2m_3+2< m_4 - 4n +2 ( 2n-2-m_4)=-4-m_4<0$, so
$\Lambda_4 \cap W = \emptyset$ and $m'_4=  m_4\leq m_1'. $ Since $n' = 2m_1'+1$, in order to compute the dimension of $ (I_{\Lambda'})_2$,  we apply Theorem \ref{teorema1} and we obtain:
$$\dim (I_{\Lambda'})_2 = \max \left \{  {n'+2\choose 2} - \sum _{i=1}^s {m'_i+2 \choose 2} ; 0
\right \} $$
where $m'_1 = m_i$ for $i >3$.

In case (ii)(c) with $s=3$, as we noted above, by Theorem \ref{teorema1}  we have:
$$\dim (I_{\Lambda'})_2 = \max \left \{  {n'+2\choose 2} - \sum _{i=1}^3 {m'_i+2 \choose 2} ; 0
\right \} $$
If we prove that  for $s\geq4$
\begin{equation}\label{s>3}
{n+2\choose 2} - \sum _{i=1}^s {m_i+2 \choose 2}
+\sum _ {i,j=1,2,3} {m_i+m_j-n+2 \choose 2}
\end{equation}
$$= {n'+2\choose 2} - \sum _{i=1}^s {m'_i+2 \choose 2}
$$

and for $s=3$:
\begin{equation}\label{s=3}
{n+2\choose 2} - \sum _{i=1}^3 {m_i+2 \choose 2}
+\sum _ {i,j=1,2,3} {m_i+m_j-n+2 \choose 2} \end{equation}
$$
={n'+2\choose 2} - \sum _{i=1}^3 {m'_i+2 \choose 2}
$$
we are done. Since for $i >3$ we have  $m'_1 = m_i$, if  the equality (\ref{s=3}) holds, then
(\ref{s>3}) holds.
Recall that
$
n'=4n-2m_1-2m_2-2m_3-3 $ and
$m'_1=  m'_2 = m'_3  = 2n-m_1-m_2-m_3-2.
$
Now  (\ref{s=3})  holds if and only if
$$(n+2)(n+1)-(m_1+2)(m_1+1)-(m_2+2)(m_2+1)-(m_3+2)(m_3+1)$$ $$
+ (m_1+m_2-n+2) (m_1+m_2-n+1)+ (m_1+m_3-n+2) (m_1+m_3-n+1)$$
$$+ (m_2+m_3-n+2) (m_2+m_3-n+1)$$
$$- (2(2n-m_1-m_2-m_3)-1)(2(2n-m_1-m_2-m_3)-2)$$
$$+3(2n-m_1-m_2-m_3)(2n-m_1-m_2-m_3-1) =0,$$
and this equality is easy to check with a direct computation.

\medskip

(iii)  Since $v \leq 2m_1-n+1 \leq m_1$, we have $v-1 \leq 2m_1 - n < m_1$ and so $\dim V = v-1$.  Projecting from $V$ onto $H$ (as in Remark \ref{rem}) we obtain
$$ \Lambda' = \Lambda'_1 +\dots + \Lambda' _s  \subset H \simeq  \Bbb P^{n'}; \ \ \ \
 \Lambda'_i \simeq  \Bbb  P^{m_i'}
$$
$$n'= n-v; \ \ \ m'_1=  m_1-v;
$$
$$ m'_i=  n-m_1-1,\ \ {\rm for} \ \ 2 \leq i \leq \tau; \ \ \ \ m'_i=  m_i,\ \ {\rm for} \ \  i > \tau.$$
(Recall that  for $i >\tau$, we have  $\Lambda_1 \cap \Lambda_i  = \emptyset $).

It is not difficult to check that $n' > m'_i+m'_j $, for $i \neq j$ i.e. the generic linear spaces $\Lambda_i^\prime$ do not intersect.  So we may use Theorem \ref{teorema1} in order to compute $\dim (I_{\Lambda'})_2$ and we get
$$
\dim (I_{\Lambda'})_2 = \max\{ {n'+2\choose 2} - \sum _{i=1}^s {m'_i+2 \choose 2}; 0 \} .
$$

If  we show that
\begin{equation}\label{ultima}
 {n+2\choose 2} - \sum _{i=1}^s {m_i+2 \choose 2}
+\sum _{i=2}^{\tau}  {m_1+m_i-n+2 \choose 2}
\end{equation}
$$=
{n'+2\choose 2} - \sum _{i=1}^s {m'_i+2 \choose 2},
$$
then the proof of part $(3)$ of the theorem is complete. A long, and tedious, calculation leads us to prove that
$$(n+2)(n+1)-  (m_1+2)(m_1+1) $$
$$+ \sum _{i=2}^{\tau} (-(m_i+2) (m_i+1)+(m_1+m_i-n+2)(m_1+m_i-n+1) )$$
$$- (n-v+2)(n-v+1) + (m_1-v+2)(m_1-v+1) +\tau (n-m_1+1)(n-m_1) =0,
$$
and this can be checked by an easy direct computation.
\end{proof}

Finally we can summarize the results of the paper in a theorem
giving a complete description of $I_2(\Lambda)$, and hence of
$HF(\Lambda,2)$, for a generic configuration of linear spaces
$\Lambda$.
\begin{thm}\label{summarythm}
Let $m_1 \geq \dots \geq m_s \geq 0$.  Consider a generic
configuration of linear spaces of weight vector $(m_1, \dots,
m_s)$
$$ \Lambda = \Lambda_1 + \dots + \Lambda_s \subset \Bbb P^n,$$
and let $I_\Lambda$ be its defining ideal.

If $m_1+m_2<n$, then

$$\dim (I_{\Lambda})_2 = \max \left \{  {n+2\choose 2} - \sum _{i=1}^s {m_i+2 \choose 2} ; 0
\right \} .$$

If $m_1+m_2\geq n$, we let
$$\tau = \max \{ i \in \Bbb N | m_1+m_i \geq n \},
$$
$$v  = \sum _{i=2}^{\tau} (m_1+m_i -n+1),
$$
then the following statements hold:
\begin{enumerate}
\item[(i)] if $v \geq m_1+1$, then $\dim (I_{\Lambda})_2
=0$;

\item[(ii)] if $v$ is such that $2m_1-n+2 \leq v \leq m_1$, and
\begin{enumerate}

\item if  $\tau \geq 4$, then $\dim (I_{\Lambda})_2
=0$;

\item if  $\tau =3$, $s \geq 4$, $2n \leq  \sum
_{i=1}^{4} m_i +2$, then $\dim (I_{\Lambda})_2 =0$;

\item if  $\tau =3$,  $s \geq 4$ and $2n \geq  \sum
_{i=1}^{4} m_i +2$, or $\tau =s=3$,  then
$$\dim (I_{\Lambda})_2 $$
$$
=\max  \left \{  {n+2\choose 2} - \sum _{i=1}^s {m_i+2 \choose 2}
+\sum _ {i,j=1,2,3} {m_i+m_j-n+2 \choose 2}  ; 0
\right \} .$$

\end{enumerate}

\item[(iii)] If $v \leq 2m_1-n+1 $,   then
$$\dim (I_{\Lambda})_2 $$
$$=\max  \left \{  {n+2\choose 2} - \sum _{i=1}^s {m_i+2 \choose 2}
+\sum _{i=2}^{\tau}  {m_1+m_i-n+2 \choose 2}  ; 0
\right \} .$$
\end{enumerate}
\end{thm}
\qed

\section{An application: decomposition of polynomials}\label{applicationsection}

In this section we will consider the problem of writing homogeneous polynomials in a special way  (see Remark \ref{polirem}).  A classical result in this direction says that a quadratic form in $n+1$ variables can always be written as the sum of at most $n+1$ squares of linear forms.

We consider the rings $S=\mathbb{C}[x_0,\ldots,x_n]$ and $T=\mathbb{C}[y_0,\ldots,y_n]$, and we denote by $S_d$ and $T_d$ their homogeneous pieces of degree $d$. We consider $T$ as an
$S$-module by letting the action of $x_i$ on $T$ be that of partial differentiation with respect to $y_i$. We also use some basic notions about apolarity (for more on this see \cite{Ge,IaKa}).

Let $I\subset S$ be a homogeneous ideal and denote by $I^\perp\subset T$ the submodule of $T$ annihilated by every element of $I$. We recall that $(I_d)^\perp=(I^\perp)_d$.

Given linear forms $l_{i,j}\in T_1,i=1,\ldots,r,j=0,\ldots,n_i$ we
ask the question $(\star)$:

{\it\noindent does the following vector space equality
hold
$$T_d=\left(\mathbb{C}[l_{1,0},\ldots,l_{1,n_1}]\right)_d+\ldots
+\left(\mathbb{C}[l_{s,0},\ldots,l_{s,n_s}]\right)_d,$$ where
$\left(\mathbb{C}[l_{i,0},\ldots,l_{i,n_i}]\right)_d$ is the
degree $d$ part of the subring of $T$ generated by the $l_{i,j}$'s
for a fixed $i$}.

\medskip

\begin{rem}\label{polirem} Notice that if $(\star)$ has an affirmative answer then any degree $d$ form
$f(y_0, \dots, y_n)$ can be written as
$$
f_1(l_{1,0},\ldots,l_{1,n_1}) + \ldots
+f_s(l_{s,0},\ldots,l_{s,n_s})
$$
for suitable homogeneous polynomials $f_i$.
\end{rem}

The connection with configurations of linear spaces is given by the
following results.

\begin{lem} Let $\Lambda\subset\PP^n$ be an $i$ dimensional
linear space having defining ideal $I$. Then, for any $d$, we have
the following:

$$ (I^\perp)_d =\left(\mathbb{C}[l_0,\ldots,l_i]\right)_d $$

where the linear forms $l_i\in T_1$ generate $(I^\perp)_1$.
\end{lem}
\begin{proof} Obvious, since after a linear change of variables, we may assume
$$I=(x_0,\ldots,x_{n-i-1}).$$
\end{proof}

\begin{prop}
Let $\Lambda=\Lambda_1+\ldots+\Lambda_s\subset\PP^n$ be a
configuration of linear spaces having defining ideal $I$ and such
that $\dim \Lambda_i=n_i$. Then, for any $d$, the following holds:
$$ (I^\perp)_d=\left(\mathbb{C}[l_{1,0},\ldots,l_{1,n_1}]\right)_d+\ldots +\left(\mathbb{C}[l_{s,0},\ldots,l_{s,n_s}]\right)_d$$
where the linear forms $l_{i,j}\in T_1$ are such that the degree $1$ piece of $(l_{i,0},\ldots,l_{i,n_i})^\perp$ generates the ideal of $\Lambda_i$.
\end{prop}
\begin{proof}
The proof follows readily from the previous lemma once we recall
that $(I\cap J)^\perp=I^\perp + J^\perp$.
\end{proof}

Now we can make clear the connection with question $(\star)$.
Given linear forms $l_{i,j}\in T_1,i=1,\ldots,r,j=0,\ldots,n_i$,
consider the linear spaces $\Lambda_i\subset\PP^n$ having defining
ideal generated by $(l_{i,0},\ldots,l_{i,n_i})_1^\perp$. Then any
degree $d$ element in $T$ can be written as a form in the
$l_{i,j}$'s if and only if $I_d=0$ where $I$ is the ideal of the
configuration of linear spaces
$\Lambda=\Lambda_1+\ldots+\Lambda_s$. Hence, Theorem§
\ref{summarythm} produces a complete answer to question $(\star)$
for $d=2$ and for a generic choice of linear forms $l_{i,j}$'s.


\begin{thebibliography}{MYDF08}

\bibitem[BPS05]{BjornerPeevaSidman}
A.~Bj{\"o}rner, I.~Peeva, and J.~Sidman.
\newblock Subspace arrangements defined by products of linear forms.
\newblock {\em J. London Math. Soc. (2)}, 71(2):273--288, 2005.

\bibitem[Car05]{Ca05Siena}
E.~Carlini.
\newblock Codimension one decompositions and {C}how varieties.
\newblock In {\em Projective varieties with unexpected properties}, pages
  67--79. Walter de Gruyter GmbH \& Co. KG, Berlin, 2005.

\bibitem[Car06]{Ca04JA}
E.~Carlini.
\newblock Binary decompositions and varieties of sums of binaries.
\newblock {\em J. Pure Appl. Algebra}, 204(2):380--388, 2006.

\bibitem[CC07]{CaCat07}
E.~Carlini and M.~V. Catalisano.
\newblock Existence results for rational normal curves.
\newblock {\em J. Lond. Math. Soc. (2)}, 76(1):73--86, 2007.

\bibitem[CC09]{CaCat09}
E.~Carlini and M.~V. Catalisano.
\newblock On rational normal curves in projective space.
\newblock {\em J. Lond. Math. Soc. (2)}, 80(1):1--17, 2009.

\bibitem[CCG09]{CarCatGer}
E.~Carlini, M.~V. Catalisano, and A.V. Geramita.
\newblock The {H}ilbert function of a generic configuration of lines and one
  plane.
\newblock {\em in preparation}, 2009.

\bibitem[CGG05]{MR2202248}
M.~V. Catalisano, A.~V. Geramita, and A.~Gimigliano.
\newblock Higher secant varieties of {S}egre-{V}eronese varieties.
\newblock In {\em Projective varieties with unexpected properties}, pages
  81--107. Walter de Gruyter GmbH \& Co. KG, Berlin, 2005.

\bibitem[Der07]{Derksen}
H.~Derksen.
\newblock Hilbert series of subspace arrangements.
\newblock {\em J. Pure Appl. Algebra}, 209(1):91--98, 2007.

\bibitem[DS02]{DerksenSidman}
H.~Derksen and J.~Sidman.
\newblock A sharp bound for the {C}astelnuovo-{M}umford regularity of subspace
  arrangements.
\newblock {\em Adv. Math.}, 172(2):151--157, 2002.

\bibitem[Ger96]{Ge}
A.~V. Geramita.
\newblock Inverse systems of fat points: {W}aring's problem, secant varieties
  of {V}eronese varieties and parameter spaces for {G}orenstein ideals.
\newblock In {\em The Curves Seminar at Queen's, Vol.\ X (Kingston, ON, 1995)},
  volume 102 of {\em Queen's Papers in Pure and Appl. Math.}, pages 2--114.
  Queen's Univ., Kingston, ON, 1996.

\bibitem[Har92]{Harris}
J.~Harris.
\newblock {\em Algebraic geometry, A first course}.
\newblock Graduate Texts in Math. Springer-Verlag, New York, 1992.

\bibitem[HH82]{HartshorneHirschowitz}
R.~Hartshorne and A.~Hirschowitz.
\newblock Droites en position g\'en\'erale dans l'espace projectif.
\newblock In {\em Algebraic geometry ({L}a {R}\'abida, 1981)}, volume 961 of
  {\em Lecture Notes in Math.}, pages 169--188. Springer, Berlin, 1982.

\bibitem[IK99]{IaKa}
A.~Iarrobino and V.~Kanev.
\newblock {\em Power sums, {G}orenstein algebras, and determinantal loci},
  volume 1721 of {\em Lecture Notes in Mathematics}.
\newblock Springer-Verlag, Berlin, 1999.

\bibitem[MYDF08]{DerksenApplication}
Y.~Ma, A.~Y. Yang, H.~Derksen, and R.~Fossum.
\newblock Estimation of subspace arrangements with applications in modeling and
  segmenting mixed data.
\newblock {\em SIAM Rev.}, 50(3):413--458, 2008.

\bibitem[Sid04]{Sidman04}
J.~Sidman.
\newblock Defining equations of subspace arrangements embedded in reflection
  arrangements.
\newblock {\em Int. Math. Res. Not.}, 15:713--727, 2004.

\bibitem[Sid07]{Sidman07}
J.~Sidman.
\newblock Resolutions and subspace arrangements.
\newblock In {\em Syzygies and {H}ilbert functions}, volume 254 of {\em Lect.
  Notes Pure Appl. Math.}, pages 249--265. Chapman \& Hall/CRC, Boca Raton, FL,
  2007.

\end{thebibliography}
\end{document}